\documentclass{amsart}[12pt]
\usepackage{amssymb}
\usepackage{amsmath}
\usepackage{amsthm}
\usepackage{amscd}
\usepackage[all]{xy}

\newtheorem{theorem}{Theorem}[section]
\newtheorem{lemma}[theorem]{Lemma}

\theoremstyle{definition}
\newtheorem{definition}[theorem]{Definition}

\newtheorem{proposition}[theorem]{Proposition}
\theoremstyle{remark}

\numberwithin{equation}{section}

\include{scrload}

\newcommand{\F}{\mathbb{F}}

\newcommand{\M}{\mathbb{M}}

\newcommand{\Z}{\mathbb{Z}}
\newcommand{\R}{\mathbb{R}}

\newcommand{\Q}{\mathbb{Q}}

\newcommand{\ga}{\gamma}

\newcommand{\rb}{\rangle}
\newcommand{\lb}{\langle}
\DeclareMathOperator{\Spec}{Spec}

\begin{document}
\title{Newton-Hodge Filtration for Self-dual $F$-Crystals}
\author{ N. E. Csima}
\date{}
\address{Department of Mathematics, University of Chicago, 5734 S. University Ave. Chicago, Illinois 60637}
\email{ecsima@math.uchicago.edu}

\maketitle


\newcounter{marker}
\newcounter{marker2}

\begin{abstract}
In this paper we study $F$-crystals with self-dual structure over base schemes of characteristic $p$.  We generalize Katz's Newton-Hodge Filtration Theorem  to $F$-crystals with self-dual structure.

\end{abstract}
\section*{Introduction}  
        Let $k$ be a perfect field of characteristic $p>0$, and let $K$ denote the field of fractions of the ring of Witt vectors $W(k)$.  Let $\sigma$ be the Frobenius automorphism of $K$. An $F$-crystal over $k$ is a free $W(k)$-module $M$ of finite rank endowed with a $\sigma$-linear endomorphism $F: M\to M$ which is injective.  By an $F$-isocrystal over $k$ we mean a finite dimensional $K$-vector space $V$ endowed with a $\sigma$-linear bijection $F:V\to V$. In the 1950's Dieudonn{\' e} introduced these notions and classified $F$-isocrystals in the case where $k$ is algebraically closed \cite{Manin}.  It turns out that the category of $F$-isocrystals over $k$ is semi-simple when $k$ is algebraically closed, and that the simple $F$-isocrystals are parameterized by $\Q$.  This result also classified $p$-divisible groups over $k$ up to isogeny.

\vspace{1ex}

        In the 1960's Grothendieck introduced the notion of $F$-crystals over $\F_p$-algebras, or ``$F$-crystals over schemes''.  This formalized the idea of having a family of  $F$-crystals over perfect fields of characteristic $p$  parameterized by points on the original scheme.  These came from Dieudonn{\'e } theory for $p$-divisible groups over arbitrary base schemes of characteristic $p$.  To each $F$-crystal over $k$ one can associate a Newton polygon and a Hodge polygon.  It is natural to ask, given an $F$-crystal over a scheme, how the Newton and Hodge polygons associated to each point vary over the scheme.  Grothendieck  proved that the Newton polygon of an $F$-crystal rises under specialization.  N. Katz proved the Newton-Hodge Decomposition Theorem  for $F$-crystals over $k$ and then proved a generalized version for $F$-crystals over certain $\F_p$-algebras \cite{Katz}.  This is referred to as Newton-Hodge filtration.   
\vspace{1ex}

        The $\sigma$-conjugacy classes of $GL_n(K)$ are in bijective correspondence with the isomorphism classes of $F$-isocrystals $(V,F )$ of dimension $n$.  This leads to the question of how the $\sigma$-conjugacy classes of other groups can be described, and this question was answered in \cite{Ko} and \cite{Ko2}.  The idea is to consider the category of $F$-isocrystals with some additional structure on them, and then show that the isomorphism classes in this new category characterize the $\sigma$-conjugacy classes.  For a connected linear algebraic group $G$ over $\Q_p$, the $\sigma$-conjugacy classes of $G$ (usually denoted $B(G)$) have been described by isomorphism classes of what are called $F$-isocrystals with $G$-structure. One example of this phenomenom is when $G=GSp_n(K)$.  In this case an $F$-isocrystal with $G$-structure is called a symplectic $F$-isocrystal.  It is a triple $(V, F, \lb\text{ },\text{ }\rb )$, where $\lb\text{ },\text{ }\rb$ is a non-degenerate skew symmetric form, $(V,F )$ is an $F$-isocrystal, and $\lb Fx,Fy\rb = c\sigma(\lb x, y\rb)$ for all $x,y\in V$ and some fixed $c\in K^{\times}$.   In \cite{Ko3} Kottwitz proved a group theoretic generalization of Katz's Newton-Hodge Decomposition Theorem for split connected reductive groups.  Later Viehmann improved this result in \cite{V}.  The ideas above came together in \cite{RR}.  They considered $F$-crystals over $\F_p$-algebras up to isogeny, or $F$-isocrystals over schemes.  Using Tannakian theory they were able to define $F$-isocrystals with $G$-structure over schemes.  They proved a generalization of Grothendieck's specialization theorem and Mazur's Inequality.
        
In this manuscipt we will define the notion of a self-dual $F$-crystal over $k$ and the notion of a self-dual $F$-crystal over an $\F_p$-algebra.  These definitions will be motivated by the notion of symplectic $F$-isocrystals.   Then we will state and prove the Newton-Hodge Decomposition Theorem for self-dual  $F$-crystals over $k$ (\ref{HNDS})  and state and prove a version of Newton-Hodge Filtration for self-dual  $F$-crystals over certain $\F_p$-algebras (\ref{NHFiltS}).  This filtration is self-dual, and thus this result produces the autodual divisible Hodge $F$-crystals considered in Katz \cite{K}.  Much of this work follows \cite{Katz}.   


\section{Review of $F$-crystals over perfect fields}
In this section we will review some facts about $F$-crystals.  These definitions can be found in \cite{Katz}.

Let $k$ be a perfect field of characteristic $p>0$.  Let $W(k)$ denote the ring of Witt vectors over $k$.  Note that $W(k)$ is a complete discrete valuation ring with uniformizer $p$.  We will denote the valuation on $W(k)$ by $\nu$.  We will let $\sigma: W(k)\to W(k)$ denote the frobenius automorphism, which is the lifting of the frobenius on $k$.  

\begin{definition} An $F$-crystal $(M,F)$ over $k$ is a free $W(k)$-module of rank $n$ equipped with a $\sigma$-linear injection $F:M\to M$.  To say that $F$ is $\sigma$-linear means that for $a\in W(k)$ and $m\in M$ we have $F(am) = \sigma(a)F(m)$ and  also that $F$ is an additive map.  A morphism of crystals $\varphi:(M,F)\to (M',F')$ is a map of $W(k)$-modules such that $\varphi\circ F = F' \circ \varphi$.  
\end{definition}

For an $F$-crystal $(M,F)$, given a $W(k)$-basis $\{e_1, \ldots , e_n\}$ of $M$ we can associate a matrix $A\in M_n(W(k))$ to $(M,F)$ defined by $F(e_i) = \sum\limits_{j=1}^n A_{ji}e_j$.  In this case we write $F$ as $A\circ \sigma$.  It is occasionally useful to think of an $F$-crystal as being given by a matrix.  We also note that in this case $A\in M_n( W(k))\cap GL_n(K)$.   

Given an $F$-crystal $(M,F)$ of rank $n$ we can associate to it a Hodge polygon.  We first note that since $F$ is injective, $F(M)$ is a $W(k)$-submodule of $M$ and $F(M)$ has rank $n$.  By the theory of elementary divisors we can find two $W(k)$-bases ${v_1,\ldots, v_n}$ and ${w_1,\ldots,w_n}$ of $M$ such that
\[F(v_i)=p^{a_i} w_i\]
and
\[0\leq a_1\leq a_2\leq\ldots\leq a_n.\] The $a_i$ are referred to as the Hodge slopes of $(M,F)$.

The Hodge polygon of $(M,F)$ is the graph of the function $Hodge_F :[0,n]\to \R$ defined on integegrs by 
\[\left\{ \begin{array}{ll}
        Hodge_F(i)=0 & \text{ if }i=0 \\
  Hodge_F (i)=a_1+\ldots +a_i & \text{ if }i = 1,\ldots , n 
\end{array}\right. \]
and then extended linearly between integers.
Another characterization of the Hodge numbers of an $F$-crystal comes from the Cartan Decomposition of the group $GL_n(K)$.  Let $\mu = (a_1, \ldots ,a_n)$.  We will let $[p^{\mu}]$ denote the diagonal matrix where $[p^{\mu}]_{ii} = p^{a_i}$.  We have

\[ GL_n(K) = \coprod\limits_{\mu} GL_n(W(k))\backslash [p^{\mu}]  /GL_n(W(k)) \]
where $\mu$ ranges over $(a_1, \ldots , a_n)$ such that $a_1\leq\cdots\leq a_n$. Given a matrix $A$ for $F$ we have $A = K_1 [p^{\mu} ]K_2$ where $K_1$, $K_2 \in GL_n(W(k))$. The Hodge slopes for $F$ are the numbers $a_1\leq \ldots \leq a_n$.  


To each $F$-crystal $(M,F)$ we can also associate a Newton polygon.  Let $k'$ be an algebraic closure of $k$ and let $K'$ be the field of fractions of $W(k')$.  Consider $s/r\in\Q$ where $r$ and $s$ are relatively prime integers, with $s\geq 0$ and $r>0$.  Let \[V_{s/r} = (\Q_p[T]/(T^r - p^s)\otimes_{\Q_p} K' , (\text{mult. by }T)\otimes \sigma).\]  Note that the dual of an $F$-isocrystal is given by $(V^*, F_{V^*})$ where $F_{V^*}(\phi)(x) = \sigma(\phi(F^{-1} x))$ for $\phi\in V^*$.  For $s/r<0$ we define $V_{s/r} = (V_{-s/r})^*$. By Dieudonn{\' e} the category of $F$-isocrystals  over $k'$ is semisimple, with the simple objects being the $V_{s/r}$'s.  That is, any $F$-isocrystal $(V,F)$ over $k'$ is isomorphic to a direct sum $\bigoplus\limits_{i} V_{s_i/r_i}$.   If $V$ is $n$-dimensional the Newton slopes of $(V, F)$ are $(\lambda_1,\ldots ,\lambda_n)$ where each $s_i/r_i$ is repeated $r_i$ times and $\lambda_1\leq \cdots\leq\lambda_n$.  A proof that this decomposition exists can be found in \cite{D}.  Consider the $F$-isocrystal 
$(M\otimes_{W(k)} K',F\otimes_{W(k)} \sigma)$ which comes from the $F$-crystal $(M,F)$.  The Newton slopes of $(M\otimes_{W(k)} K',F\otimes_{W(k)} \sigma)$ are independent of the choice of  the algebraic closure $k'$ of $k$.  Thus we can speak of the Newton slopes of $(M,F)$, which are defined to be the Newton slopes of $(M\otimes_{W(k)} K',F\otimes_{W(k)} \sigma)$.  The Newton polygon of $(M,F)$
is the graph of the function $Newton_F :[0,n]\to \R$ defined on integers by
\[\left\{ \begin{array}{ll}
        Newton_F(i)=0 & \text{ if }i=0 \\
  Newton_F (i)=\lambda_1+\ldots +\lambda_i & \text{ if }i = 1,\ldots , n 
\end{array}\right. \] and then extended linearly between successive integers.  

Another useful characterization of the Newton slopes of an $F$-crystal $(M,F_M)$ is the following.  Suppose $M$ has rank $n$.  Then let $L$ be the field of fractions of the ring $W(k')[X]/(X^{n!} -p) = W(k')[p^{\frac{1}{n!}}]$.  Note that we can extend $\sigma$ to $L$ by setting $\sigma(X) = X$.  By Dieudonn{\' e} if $(M, F)$ has Newton slopes $\lambda_1\leq \cdots \leq \lambda_n$ then there is a $L$-basis $e_1, \ldots e_n$ of $M\otimes_{W(k')} L$ such that $(F\otimes\sigma)(e_i)=p^{\lambda_i}e_i$ for $i = 1,\ldots , n$ (here $p^{\lambda_i}= X^{n!\lambda_i}$).

The following result relates the two polygons.

\begin{theorem}[Mazur's Inequality] For any $F$-crystal $(M,F)$ the Newton polygon lies above the Hodge polygon.  The two polygons have the same initial point $(0,0)$ and end point $(n, \nu(det F))$.
\end{theorem}

\section{Self-dual $F$-crystals over perfect fields}

In this section we will define the notion of a self-dual $F$-crystal over $k$ where $k$ is a perfect field of characteristic $p$.  We will relate this back to the notion of symplectic $F$-isocrystals. 

\begin{definition} Let $M^{*} = Hom_{W(k)}(M,W(k))$. A self-dual $F$-crystal is a quintuple $(M,F_M,F_{M^*},\Psi , c)$ where

\begin{list}{(\arabic{marker})}{\usecounter{marker}}
\item $c\in  W(k)$ and $c\neq 0$.
\item $(M,F_M)$ and $(M^*,F_{M^*})$ are $F$-crystals over $k$.
\item $F_{M^*}\circ {F_M}^* = c\cdot id$ and ${F_M}^*\circ F_{M^{*}} = \sigma^{-1}(c) id$. Note here ${F_M}^*(\varphi ) = \sigma^{-1}\circ\varphi \circ F_M$ where $\varphi \in M^*$.  
\item $\Psi:(M,F_M)\xrightarrow{\simeq}(M^*,F_{M^*})$.
\end{list}
\end{definition}
Note that we can recover the map $F_{M^*}$ from the definition above.  Consider the $\sigma$-linear map $F_{M^*}'(\varphi) = \sigma \circ\varphi\circ\sigma^{-1}(c)F_M^{-1}$ where $\varphi\in M^*$.  Indeed we see that if $\varphi\in M^*$ then
\begin{equation}\notag\begin{split}(F_{M^*}'\circ {F_M}^*)(\varphi) &= F_{M^*}'(\sigma^{-1}\circ\varphi\circ F_M) \\
         &= \sigma\circ\sigma^{-1}\circ\varphi\circ F_M \circ\sigma^{-1}(c)F_M^{-1}\\
         &= \varphi\circ c\cdot id \circ F_M\circ F_M^{-1} \\
         &= c\cdot\varphi
\end{split}\end{equation}
Thus we've shown that $F_{M^*}'\circ {F_M}^* = c\cdot id$.  We also have (from our work above and (3)) that
\[ c\cdot id \circ F_{M^*}= F_{M^*}'\circ {F_M}^* \circ F_{M^*} = F_{M^*}'   \circ\sigma^{-1}(c)\cdot id = c  \cdot id F_{M^*}'.\]
Since $W(k)$ is torsion free this implies that $F_{M^*} = F_{M^*}'$.  In terms of matrices this means that if $A$ is a matrix for $F_M$ with respect to some $W(k)$-basis $m_1, \ldots ,m_n$ then the matrix for $F_{M^*}$ in terms of the dual basis $m_1^* , \ldots m_n^*$ is $c\cdot (A^{-1})^{t}$.

In the case where $\Psi = -\Psi^*$ we say the $F$-crystal is symplectic.  If $\Psi = \Psi^*$ we say the $F$-crystal is orthogonal.

The following proposition relates the definition of symplectic $F$-crystal to the definition of symplectic isocrystal.
\begin{proposition} %
%
If $(M,F_M,F_{M^*},\Psi, c)$ is a $2n$-dimensional symplectic $F$-crystal over $k$,  then $(M,F)$ is an $F$-crystal equipped with a non-degenerate skew-symmetric form $\lb\text{ },\text{ }\rb : M \times M \to W(k)$ such that $\lb F_M(m_1), F_M(m_2)\rb = c \sigma ( \lb m_1 , m_2 \rb )$.
\end{proposition}
\begin{proof}
Given $(M,F_M, F_{M^*},\Psi, c)$ we define $\lb\ ,\ \rb$ by $\lb m_1, m_2 \rb := \Psi(m_i)(m_2)$ for $m_1$,$m_2\in M$.  Nondegeneracy of $\lb\ ,\ \rb$ follows from the fact that $\Psi$ is an isomorphism.  Lastly, recall that by (4) we have $\Psi\circ F_M = F_{M^*}\circ \Psi$.  Thus 
\begin{equation}\notag\begin{split} \lb F_M(m_1),F_M(m_2)\rb &= \Psi(F_M(m_1))(F_M(m_2))\\
        &=(F_{M^*}(\Psi(m_1))(F_M(m_2)) \text{ (since }\Psi\text{ is a morphism of crystals)} \\
        &=\sigma(\Psi(m_1)(\sigma^{-1}(c)\cdot{F_M}^{-1}(F_M(m_2))))\text{( by the formula for }F_{M^*}\text{)}\\
        &=\sigma(\Psi(m_1)(\sigma^{-1}(c)\cdot m_2))\\
        &=\sigma(\sigma^{-1}(c)\Psi(m_1)(m_2))\\
        &=c\cdot \sigma(\Psi(m_1)(m_2))\\
        &=c\sigma(\lb m_1,m_2 \rb)      \\
\end{split}\end{equation}
as required.

\end{proof}


Let $\lambda_1\leq \cdots\leq\lambda_{n}$ be the Newton slopes of $(M,F_M,F_{M^*},\Psi, c)$, a self-dual $F$-crystal,  and let $a_1\leq \cdots\leq a_n$ be the Hodge slopes of  $(M,F_M,F_{M^*},\Psi, c)$.  Let $(V^*, F_{V^*})$ denote the dual of $(M\otimes K , F_M \otimes\sigma )$. We see that the $(M^*\otimes K , F_{M^*}\otimes \sigma)$ $ = (V^*, \sigma^{-1}(c)F_{V^*})$ because of the relationship between $F_M$ and $F_{M^*}$. Thus the Newton slopes of $(M^*, F_{M^*})$ are $\nu(c)-\lambda_{n} \leq \ldots\leq \nu(c) - \lambda_1$.  Since $(M,F_M)$ and $(M^*,F_{M^*})$ are isomorphic they have the same Newton slopes, so we have for each $i$, $\lambda_i + \lambda_{n-i+1} = \nu(c)$.  

We believe the following proposition was known previously but we were unable to find a proof in the literature.  The author would like to thank M. Boyarchenko for providing the following proof.

\begin{proposition} Let $a_1, \ldots ,a_{n}$ be the  Hodge slopes of $(M,F_M,F_{M^*},\Psi, c)$.  Then $a_i + a_{n-i+1} = \nu(c)$.  
\end{proposition}
\begin{proof}  We will prove two lemmas which will imply the result.  Let $V = M\otimes K$ and let $\lb\ ,\ \rb = \Psi(\ )(\ )$.  Let $N\subset V$ be an $n$-dimensional $W(k)$-lattice.  Then by the theorem of elementary divisors there exists a basis $e_1, \ldots ,e_n$ of $M$ such that $p^{a_1}e_1,\ldots ,p^{a_n}e_n$ (with $a_1\leq\cdots\leq a_n$ ) is a basis of $N$. We will refer to $a_1, \ldots, a_n$ as the invariants of $N$ with respect to $M$.  Let $N^{\perp} = \{ x\in V | \lb x, y\rb \in W(k)\ \forall y\in N\}$.

\begin{lemma} If $a_1,\ldots, a_n$ are the invariants of $N$ with respect to $M$ then $-a_1, \ldots, -a_n$ are the invariants of $N^\perp$ with respect to $M$.
\end{lemma}
\begin{proof} Suppose $e_1,\ldots,e_n$ is a basis of $M$ such that $p^{a_1}e_1,\ldots, p^{a_n}e_n$ is a basis of $N$.  Let $e_1^*, \ldots,e_n^*$ be the basis of $M$ such that $\lb e_i^*, e_j\rb = \delta_{ij}$.  Such a basis exists since $\Psi$ is an isomorphism.  Let $\Lambda = \text{Span}_{W(k)}\{p^{-a_1}e_1^*,\ldots ,p^{-a_n}e_n^*\}$.  We claim that $N^{\perp} = \Lambda$.  Clearly $\Lambda \subset N^\perp$.  Suppose that $n\in N^\perp$.  Then $n = b_1e_1^* + \cdots + b_n e_n^*$ for some $b_i\in K$ since $e_1^* , \ldots ,e_n^*$ is a basis of $V$.   Now for each $i = 1, \ldots, n$ we have $\lb n, p^{a_i} e_i\rb = b_i p^{a_i}\in W(k)$ since $n\in N^\perp$.  Thus $n = b_1p^{a_1} \cdot p^{-a_1}e_1^*  + \cdots + b_np^{a_n} \cdot p^{-a_n}e_n^* \in \Lambda$. So we've shown that $p^{-a_1}e_1^*$,$\ldots$,$p^{-a_n}e_n^*$ is a basis for $N^{\perp}$, as required. 
\end{proof}
\begin{lemma} Consider $F_M(M)\subset V$.  Then $F_M(M)^{\perp} = \frac{1}{c}F_M(M)$.
\end{lemma}
\begin{proof} Recall that $\lb F_M(m_1), F_M(m_2) \rb = c\sigma(\lb m_1 , m_2 \rb ) $ for all $m_1$, $m_2 \in M$.  We have 

\begin{equation}\notag\begin{split} m\in F_M(M)^\perp &\Leftrightarrow \lb m, F_M(m_1)\rb\in W(k)\ \forall m_1\in M \\
        &\Leftrightarrow c\sigma(\lb F_M^{-1}(m) , m_1\rb ) \in W(k)\ \forall m_1 \in M \\
        &\Leftrightarrow \sigma^{-1}(c)\lb F_M^{-1}(m) , m_1\rb \in W(k)\ \forall m_1\in M \\
        &\Leftrightarrow \lb F_M^{-1}(cm), m_1\rb \in W(k)\ \forall m_1 \in M \\
        &\Leftrightarrow F_M^{-1}(cm)\in M^{\perp} = M  \\
        &\Leftrightarrow cm \in F_M(M)  \\
        &\Leftrightarrow m \in \frac{1}{c} F_M(M)\\
        \end{split}\end{equation}
\end{proof}

 Let $a_1, \ldots, a_n$ be the Hodge slopes of  $(M,F_M,F_{M^*},\Psi, c)$.  We note that this is the same as the invariants of $F(M)$ with respect to $M$.  We see from the two lemmas that the invariants of $F(M)$ with respect to $M$ are $\nu(c) - a_n \leq\cdots \leq \nu(c) - a_1$.
\end{proof}

Thus we see that for self-dual $F$-crystals the Newton and Hodge slopes satisfy a symmetry condition.

\section{Newton-Hodge Decomposition for Self-Dual $F$-crystals}

Our goal is to prove a version of Newton-Hodge Decomposition for self-dual $F$-crystals.  We continue to assume throughout this section that $k$ is a perfect field of characteristic $p$.  The result for $F$-crystals is:

\begin{theorem}[Newton-Hodge Decomposition, \cite{Katz}, Thm. 1.6.1]\label{HNDecomp} Let $(M,F)$ be an $F$-crystal of rank $n$.  Let $(A,B)\in\Z\times\Z$ be a break point of the Newton polygon of $(M,F)$ which also lies on the Hodge polygon of $(M,F)$.  Let $\lambda_1 \leq  \cdots\leq \lambda_n$ denote the Newton slopes of $(M,F)$ and let $a_1\leq \cdots \leq a_n$ be the Hodge slopes.  Then there exists a unique decomposition of $(M,F)$ as a direct sum 
\[ (M,F) = (M_1\oplus M_2, F_1\oplus F_2)\]
of two $F$-crystals $(M_1,F_1)$ and $(M_2,F_2)$ such that 
\begin{list}{(\arabic{marker})}{\usecounter{marker}}
                \item $M_1$ has rank $A$, the Hodge slopes of $(M_1, F_1)$ are  $a_1 \leq \cdots \leq a_A$ and the Newton slopes of $(M_1, F_1)$ are $\lambda_1 \leq \cdots \leq \lambda_A$.
                \item $M_2$ has rank $n-A$, the Hodge slopes of $(M_2, F_2)$ are  $h_{A+1} \leq \cdots \leq h_n$ and the Newton slopes of $(M_2, F_2)$ are $\lambda_{A+1} \leq \cdots \leq \lambda_n$.
\end{list}
\end{theorem} 

In terms of matrices the theorem above states that it is possible to find a $W(k)$-basis of $(M, F)$ such that the matrix for $F$ is $\begin{bmatrix} F_1 & 0 \\ 0 & F_2 \end{bmatrix}$ where $(M_1, F_1)$ and $(M_2 ,F_2)$ satisfy the slope conditions listed above.  Here  $F_1 := F|_{M_1}$ and $F_2 := F|_{M_2}$.  We also have that $M_1$ is the unique $F$-stable submodule of $M$ such that $M_1$ is free of rank $A$, the Newton slopes $M_1$ are $\lambda_1\leq \ldots \leq\lambda_A$ , and $M_2 = M/M_1$ is free of rank $n-A$ (\cite{Katz}, 1.6.3).

The result below is the generalization of Newton-Hodge Decomposition to the case of self-dual $F$-crystals.  We note that in  \cite{Ko3} Kottwitz proved a group theoretic generalization of \ref{HNDecomp} for split connected reductive groups in the case where $k$ is algebraically closed.  Later in \cite{V} it was noted that the result in \cite{Ko3} could be strengthened using essentially the same proof as in \cite{Ko3}.  Here our result is not in purely group theoretic terms and there are less assumptions on $k$ (we only assume that $k$ is perfect, but it need not be algebraically closed). 
 
\begin{theorem}[Newton-Hodge Decomposition for Self-Dual $F$-crystals]\label{HNDS} Let \\ $(M,F_M,F_{M^*},\Psi, c)$ be a self-dual $F$-crystal of rank $n$.  Let $(A,B)\in\Z\times\Z$ be a break point of the Newton polygon of $(M,F_M,F_{M^*},\Psi,c)$ which also lies on the Hodge polygon of $(M,F_M,F_{M^*},\Psi,c)$ and assume $A<\frac{n}{2}$.  Let $\lambda_1 \leq  \cdots\leq \lambda_{n}$ denote the Newton slopes of $(M,F_M,F_{M^*},\Psi,c)$ let $a_1\leq \cdots \leq a_{n}$ be the Hodge slopes.  Then there exists a unique decomposition of $(M,F)$ as a direct sum 
\[ (M,F) = (M_{S_1}\oplus M_{S_2}, F_{S_1}\oplus F_{S_2})\]
of two self-dual $F$-crystals $(M_{S_1},F_{M_{S_1}},F_{M_{S_1}^*},\Psi_{S_1},c)$ and $(M_{S_2},F_{M_{S_2}},F_{M_{S_2}^*},\Psi_{S_2},c)$ such that 
\begin{list}{(\arabic{marker})}{\usecounter{marker}}
  \item The rank of $M_{S_1}$  is $2A$, the Hodge slopes of $(M_{S_1},F_{S_1})$ are $a_1 \leq \cdots \leq a_A \leq a_{n-A+1} \leq \cdots \leq a_{n}$ and the Newton slopes of $(M_{S_1},F_{S_1})$ are $\lambda_1 \leq \cdots \leq \lambda_A \leq \lambda_{n-A+1}\leq \cdots \leq \lambda_{n}$.
  \item The rank of $M_{S_2}$ is $n-2A$, the Hodge slopes of $(M_{S_2},F_{S_2})$ are $a_{A+1} \leq \cdots \leq a_{n-A}$ and the Newton slopes of $(M_{S_2},F_{S_2})$ are $\lambda_{A+1} \leq \cdots \leq \lambda_{n-A}$.
\end{list}

\end{theorem} 

\begin{proof} For a self-dual $F$-crystal $(M,F_M, F_{M^*},\Psi ,c)$, if $(A,B)$ is a break point of the Newton polygon and lies on the Hodge polygon then $(n-A,\nu(c)(\frac{n}{2}-A) - B)$ is also a break point of the Newton Polygon which lies on the Hodge polygon of $(M,F_M,F_{M^*},\Psi, c)$ by the symmetry property we showed earlier for self-dual $F$-crystals.   Thus we can apply \ref{HNDecomp} to $(M,F_M)$ twice (at the points $(A,B)$ and $(n-A, (\frac{n}{2}-A)\nu(c) - B))$ to obtain the decomposition $(M_1 \oplus M_2 \oplus M_3, F_1 \oplus F_2 \oplus F_3)$ where 
\begin{list}{(\arabic{marker})}{\usecounter{marker}}
  \item  The rank of $M_1$ is $A$, the Hodge slopes of $(M_1,F_1)$ are $a_1 \leq \cdots \leq a_A$ and  the Newton slopes of $(M_1,F_1)$ are $\lambda_1 \leq \cdots \leq \lambda_A $.
  \item The rank of $M_2$ is $n-2A$, the Hodge slopes of $(M_2,F_2)$ are $a_{A+1} \leq \cdots \leq a_{n-A}$ and the Newton slopes of $(M_2,F_2)$ are $\lambda_{A+1} \leq \cdots \leq \lambda_{n-A}$
  \item The rank of $M_3$ is $A$, the Hodge slopes of $(M_3,F_3)$ are $a_{n-A+1} \leq \cdots \leq a_{n}$ and the Newton slopes of $(M_3,F_3)$ are $\lambda_{n-A+1} \leq \cdots \leq \lambda_{n}$.
\end{list}

We will show that the $F$-crystals $(M_1 \oplus M_3 ,F_1\oplus F_3)$ and $(M_2, F_2)$ both have a self-dual structure which comes from the restriction of $\Psi$ to $M_1\oplus M_3$ and $M_2$.  First we note that since $(M_1 \oplus M_2 \oplus M_3, F_1 \oplus F_2 \oplus F_3)$ there exists a basis $m_1,\ldots ,m_{n}$  for $M$ such that the matrix for $F_M$ is 
\[\begin{bmatrix} F_1 & & 0 \\ & F_2 & \\ 0 & & F_3 \end{bmatrix}. \]
Here $F_1 = F|_{M_1}$, $F_2 = F|_{M_2}$ and $F_3 = F|_{M_3}$. Thus the matrix for $F_{M^*}$ (by our remark in section 2) in terms of the dual basis $m_1^* ,\ldots ,m_n^*$ is
\[\begin{bmatrix} c( F_1^{-1})^t & & 0\\ & c(F_2^{-1})^t & \\ 0 & & c(F_3^{-1})^t \end{bmatrix}.\] 
Thus we see that $(M_1^* \oplus M_2^* \oplus M_3^*, F_{M^*}|_{M_1^*} \oplus F_{M^*}|_{M_2^*} \oplus F_{M^*}|_{M_3^*}) = (M^*, F_{M^*})$ as $F$-crystals.  To show that the restriction of $\Psi$ provides the required self-dual structure we will show that $\Psi(M_1)\subset M_3^*$, $\Psi(M_2)\subset M_2^*$ and $\Psi (M_3)\subset M_1^*$.  Since we know that $\Psi: M \to M^*$ is an isomorphism and since $M^*=M_1^*\oplus M_2^*\oplus M_3^*$ this will imply that $\Psi(M_1)=M_3^*$, $\Psi(M_2)=M_2^*$ and $\Psi(M_3)=M_1^*$.  Furthermore, since $(M^*, F_{M^*})$ is a direct sum of the $F$-crystals $(M_1^*,F|_{M_1^*})$, $(M_2^*,F|_{M_2^*})$, and $(M_3, F|_{M_3^*})$, once we've shown this we will have that $\Psi|_{M_1\oplus M_3}$ and $\Psi|_{M_2}$ are both morphisms of $F$-crystals.

To show the required inclusions we recall a characterization of the Newton slopes associated to an $F$-crystal $(M,F_M)$ introduced in section 1. Let $L$ be the field of fractions of $W(k')[p^{1/n!}]$ where $k'$ is an algebraic closure of $k$.  Then there exists an $L$ basis $e_1 , \ldots , e_n$ of $M\otimes_{W(k)} L$ such that $(F_M\otimes\sigma) (e_i)= p^{\lambda_i}e_i$.  Note that since $M$ is a $W(k)$-module which is free of finite rank $Hom_{W(k)}(M,W(k))\otimes_{W(k)} L = Hom_{W(k)}(M, W(k)\otimes L) = Hom_L(M\otimes L, L)$.  Thus $\Psi: M\to M^*$ induces a non-degenerate bilinear form $\Psi\otimes\sigma : M\otimes L \to (M\otimes L)^*$. Sometimes we will refer to the bilinear form induced by $\Psi$ as $\lb\text{ } ,\text{ } \rb_M$ and the form induced by $\Psi\otimes\sigma$ as  $\lb\text{ } ,\text{ } \rb_{M\otimes L}$ to avoid confusion.  We note that in this case $\lb m_1\otimes 1, m_2\otimes 1\rb_{M\otimes L} = \lb m_1 ,m_2\rb_M$ for $m_1$, $m_2 \in M$. 

We will now show that $\Psi (M_1)\subset M_3^*$.  This is equivalent to showing that $\lb m_1,m \rb_M=0$ for all $m_1 \in M_1$ and $m\in (M_1\oplus M_2)\otimes L$. Note that $e_1 , \dots e_A$ is an $L$-basis for $M_1\otimes L$ and $e_1 ,\ldots ,e_{n-A}$ is and $L$-basis for $M_1\oplus M_2$.  This is because $(A,B)$ and $(n -A, \nu(c)(\frac{n}{2}-A) - B)$ are both breakpoints of the Newton polygon of $(M,F_M)$.  Suppose that $m_1\in M_1$ and $m\in M_1\oplus M_2$.  Then $m_1\otimes 1 = a_1 e_1 + \cdots + a_A e_A$ and $m = b_1 e_1 + \cdots + b_{n-A} e_{n-A}$ for some $a_i, b_i \in L$.  We have for any $i$ and $j$ that

\[\lb (F_M\otimes\sigma)(e_i),(F_M\otimes\sigma)(e_j)\rb_{M\otimes L} = \lb p^{\lambda_i}e_i ,p^{\lambda_j}e_j \rb_{M\otimes L} = p^{\lambda_i + \lambda_j}\lb e_i,e_j\rb_{M\otimes L}\]
 and 
\[\lb (F_M\otimes\sigma)(e_i),(F_M\otimes\sigma)(e_j)\rb_{M\otimes L} = c\sigma(\lb e_i ,e_j\rb_{M\otimes L}).\]

Thus $p^{\lambda_i + \lambda_j} \lb e_i , e_j\rb_{M\otimes L} = c\sigma(\lb e_i,e_j\rb_{M\otimes L})$.  Thus we must have $\nu(c) = \lambda_i + \lambda_j$ or $\lb e_i , e_j \rb_{M\otimes L} = 0$.  Now suppose that $1\leq i \leq A$ and $1\leq  j \leq n - A$.  Then $A+1\leq n -j +1 \leq n$, so $i < n - j +1$.  By our hypothesis, $(A, B)$ was a break point of the Newton polygon of $(M,F_M)$, so $\lambda_A <\lambda_{A+1}$ which implies $\lambda_i < \lambda_{n- j +1}$.  Thus $\lambda_i + \lambda_j < \lambda_{n-j +1} + \lambda_j = \nu(c)$.  So in this case we must have $\lb e_i , e_j\rb_{M\otimes L} = 0$.  This implies that $\lb m_1 \otimes 1 , m\otimes 1\rb_{M\otimes L} = 0 $ which is what we wanted to show.  Similar arguments show that $\Psi(M_2)\subset M_2^*$ and $\Psi(M_3)\subset M_1^*$.



\end{proof}

Note that the theorem above gives for a self-dual $F$-crystal $(M, F_M, F_{M^*}, \Psi, c)$ a filtration of $F_M$-stable submodules $0\subset M_1 \subset M_2\subset M$ where
\begin{list}{(\arabic{marker})}{\usecounter{marker}}
\item The Newton slopes of $(M_1 , F_M|_{M_1})$ are the first $A$ Newton slopes of $(M , F_M)$.
\item The Newton slopes of $(M_2 , F_M|_{M_2})$ are the first $n-A$ Newton slopes of $M$.
\item $M_1$ and $M_2$ are perpendicular, i.e. $\Psi(m_1)(m_2) = \Psi(m_2)(m_1) =0$ for all $m_1\in M_1$ and $m_2\in M_2$. 

\end{list}
\section{F-crystals over $\F_p$-algebras}
In this section we review the definition of $F$-crystals over an $\F_p$-algebra.  Again we follow \cite{Katz}.

 \begin{definition} An \emph{absolute test object} is a triple $(B,I,\ga)$ where $B$ is a $p$-adically complete and separated $\Z_p$-algebra, $I\subset B$ is a closed ideal with $p\in I$ and $\ga =\{\ga_n\}$ is a divided power structure on $I$ such that $\ga_n (p)$ is the image of $p^n /n!$ in $B$ (note that $p^n /n!\in \Z_p$).  If $A_\circ$ is an $\F_p$-algebra, then an \emph{$A_\circ$-test object} $(B,I,\ga;s)$ is an absolute test object $(B,I,\ga)$ with the structure of an $A_\circ$-algebra on $B/I$ given by $s$, i.e. $s:A_\circ\to B/I$ is a homomorphism of $\F_p$-algebras.  A morphism of $A_\circ$-test objects $f:(B,I,\ga ;s)\to (B',I',\ga ';s')$ is an algebra homomorphism $f:B\to B'$ which maps $I$ to $I'$, commutes with $\ga$ and $\ga '$ (i.e. $f\circ\ga_i = \ga_i '\circ f$ for each $i\geq 0$) and induces an $A_\circ$-homomorphism $B/I\to B'/I'$. 
\end{definition}

\begin{definition} A \emph{crystal} $M$ on $A_\circ$ assigns to every $A_\circ$-test object $(B,I,\ga ;s)$ a $p$-adically complete and separated $B$-module $M(B,I,\ga ;s)$.  It also assigns to every map $f:(B,I,\ga ;s)\to (B',I',\ga ';s)$ of $A_\circ$-test objects a $B'$-module isomorphism $M(f): M(B,I,\ga ;s)\hat{\otimes}_B B'\xrightarrow{\cong} M(B',I',\ga ';s')$, and this rule is compatible with compositions of maps of $A_\circ$-test objects.   A morphism of crystals $u:M\to N$ on $A_\circ$ assigns to each $A_\circ$-test object $(B,I,\ga ; s)$ a $B$-module map $u(B,I,\ga ;s):M(B,I,\ga ;s)\to N(B,I,\ga ;s)$ which is compatible with the maps $M(f)$ and $N(f)$.  A morphism of crystals $u:M\to N$ is an \emph{isogeny} if there exists a map $v:N\to M$ and $n\geq 0$ such that $u\circ v=p^n$ and $v\circ u=p^n$.
\end{definition}

Given a homomorphism of $\F_p$-algebras $\phi :A_\circ\to B_\circ$ and a $B_\circ$-test object $(B,I,\ga ;s)$, $(B,I,\ga ;s\circ\phi)$ is a $A_\circ$-test object.  With this in mind, given a crystal $M$ on $A_\circ$, we can define a crystal $M^{(\phi)}$ on $B_\circ$ by $M^{(\phi)}(B,I,\ga ;s) = M(B,I,\ga ;s\circ\phi)$.  Sometimes this is referred to as the ``inverse image'' crystal on $B_\circ$. If we are given a morphism $u:M\to N$ of crystals on $A_\circ$ then we obtain $u^{(\phi)}:M^{(\phi)}\to N^{(\phi)}$, a morphism of crystals on $B_\circ$ defined by $u^{(\phi)}(B,I,\ga ;s)=u(B,I,\ga ;s\circ\phi)$.

A crystal $M$ is said to be locally free of rank $n$ if for any $A_\circ$-test object $(B,I,\ga ; s)$, we have that $M(B,I,\ga ;s)$ is a locally free $B$-module of rank $n$. We are now ready to introduce the generalized version of an $F$-crystal.  

\begin{definition}An \emph{$F$-crystal $(M,F)$ on $A_\circ$} is a locally free crystal $M$ on $A_\circ$ together with an isogeny $F:M^{(\sigma )}\to M$.  A morphism between $F$-crystals $(M,F)$ and $(M',F')$ is a morphism $f:M\to M'$ of crystals on $A_\circ$ where $F'\circ f^{(\sigma)}=f\circ F$.  
\end{definition}

If $A_\circ$ is a perfect ring then the $A_\circ$-test object $(W(A_\circ),(p),\ga ;s)$ is an initial object in the category of all $A_\circ$-test objects [G].  The divided power structure $\ga$ is uniquely determined in this case by the requirement $\ga_n(p) = p^n/n!$ and the homomorphism $s: A_\circ \to W(A_\circ)/(p)$ is the inverse of the isomorphism $W(A_\circ)/(p)\xrightarrow{\cong} A_\circ$.  Evaluating at this initial object provides an equivalence of categories between the category of crystals on $A_\circ$ and the category of $p$-adically complete and separated $W(A_\circ)$-modules.  We also see that if $A_\circ$ is a perfect field then an $F$-crystal over $A_\circ$ corresponds to the definition we had before in section 2.

The operation of taking the inverse image of a crystal corresponds to extension of scalars, ie. if $M:= \M(W_{A_\circ}, (p), \ga; s)$ where  $\M$ is a crystal over $A_{\circ}$ and we have $\phi: A_\circ \to B_\circ $, a morphism of $\F_p$-algebras, then $\M^{(\phi)}(W_{B_\circ} ,(p), \ga ; s\circ\phi) := M\hat{\otimes}_{W(A_\circ)}W(B_\circ)$ where the $W(B_\circ)$-module structure comes from $W(\phi):W(A_\circ)\to W(B_\circ)$.  Given an $F$-crystal over $A_\circ$ we can also associate to each $\mathfrak{p}\in\Spec A_\circ$ an $F$-crystal over a perfect field.  Let $A_\circ^{\text{perf}}$ denote the perfection of $A_\circ$, i.e. $A_\circ^{\text{perf}} := \varinjlim\limits_{\sigma} A_\circ$.  We see by inspection that $\Spec A_\circ = \Spec A_\circ^{\text{perf}}$ as sets.  Let $\mathfrak{p}\in\Spec A_\circ$.  Let  $k_{\mathfrak{p}} = Fr(A_\circ^{\text{perf}} /\mathfrak{p})$.  By taking the inverse image on $k_{\mathfrak{p}}$ using $\phi_{\mathfrak{p}}: A_\circ \to k_{\mathfrak{p}}$ we get an $F$-crystal over $k_{\mathfrak{p}}$, namely $(M,F)^{(\phi_{\mathfrak{p}})}$.  The Hodge and Newton polygon of this $F$-crystal does not depend on the choice of $k_{\mathfrak{p}}$ only $\text{Ker}(\phi_{\mathfrak{p}})=\mathfrak{p}$.  So we may speak of what happens at a point of $\Spec A_\circ$, or what happens to $(M,F)$  ``pointwise''.

From now on we will consider the case in which $A_\circ$ is an $\F_p$-algebra which is not necessarily perfect, but one of the following two types:
\begin{list}{(\arabic{marker})}{\usecounter{marker}}
  \item $A_\circ$ is smooth over a perfect subring $A_{\circ\circ}$ of $A_{\circ}$
  \item $A_\circ$ is the ring of power series over a perfect ring $A_{\circ\circ}$ in finitely many variables
\end{list}

In both of these cases there exists a $p$-adically complete and separated $\Z_p$-algebra $A_{\infty}$ such that $A_{\infty}$ is flat over $\Z_p$, $A_{\infty}/p A_{\infty}\cong A_\circ$ and for each $n\geq 0$, $A_n := A_{\infty}/p^{n+1} A_{\infty}$ is formally smooth over $\Z/p^{n+1}\Z$ \cite{Katz}.  There is an explicit construction of such a ring in \cite{BM}.  $A_{\infty}$ should be thought of as a generalized version of a Cohen ring over a perfect field.  By \cite{G} $A_{\infty}$ is naturally a $W(A_{\circ\circ})$-algebra.  $A_{\infty}$ is unique up to automorphisms that are the identity on $W(A_{\circ\circ})$ and which reduce mod $p$ to the identity.  Using formal smoothness of the $A_n$'s we can lift $\sigma : A_\circ\to A_\circ$ to a ring homomorphism  $\Sigma: A_{\infty}\to A_{\infty}$ (i.e. $\Sigma$ reduces mod $p$ to $\sigma$), however this lifting is not unique.

Consider the $A_\circ$-test object $(A_{\infty},(p),\ga;s)$ where $s$ is the inverse of $A_{\infty}/pA_{\infty} \xrightarrow{\cong} A_\circ$.  This test object is not initial; however it is weakly initial, that is, any $A_\circ$-test object receives a map from it, though this map need not be unique.  Evaluation at this object provides an equivalence of categories between crystals on $A_\circ$ and the category of pairs $(M,\nabla)$ consisting of a $p$-adically complete and separated $A_{\infty}$-module $M$ together with an integrable, nilpotent $W(A_{\circ\circ})$-connection $\nabla$.  Fixing a lifting $\Sigma: A_{\infty}\to A_{\infty}$ of $\sigma$ we obtain an equivalence of categories between the category of $F$-crystals on $A_\circ$ and the category of triples $(M,\nabla, F_{\Sigma})$ where

\begin{list}{(\arabic{marker})}{\usecounter{marker}}
\item $M$ is a locally free $A_\infty$-module
\item $\nabla$ is an integrable, nilpotent $W(A_{\circ\circ})$-connection.
\item $F_{\Sigma}: (M^{(\Sigma)},\nabla^{(\Sigma)})\to (M,\nabla)$ is an isogeny (i.e. $F_\Sigma$ induces an isomorphism after tensoring $M^{(\Sigma)}$ and $M$ over $\Z_p$ with $\Q_p$).
\end{list}

We can also construct, given a choice of $\Sigma$, a unique inclusion $i(\Sigma):A_{\infty}\hookrightarrow W(A_\circ^{\text{perf}})$ such that $i(\Sigma)$ reduces mod $p$ to the inclusion $A_\circ\hookrightarrow A_\circ^{\text{perf}}$ and $i(\Sigma)\circ\Sigma = W(\sigma)\circ i(\Sigma)$.  Suppose we have an $F$-crystal on $A_\circ$ thought of as a triple $(M,\nabla, F_\Sigma)$.  Its inverse image on $A_\circ^{\text{perf}}$ is the pair $(M\otimes_{A_{\infty}}W(A_\circ^{\text{perf}}),F_\Sigma \otimes id)$, obtained from $(M,F_\Sigma ,\nabla)$ by extension of scalars via $i(\Sigma):A_\infty\to W(A_\circ^{\text{perf}})$.  Thus at every point $\mathfrak{p}$ of $\Spec A_\circ $ the $F$-crystal over $W(k_{\mathfrak{p}})$ associated to $(M,F_\Sigma ,\nabla_\Sigma )$  is given by $(M\otimes_{A_{\infty}} W(k_{\mathfrak{p}}) , F_{\Sigma}\otimes_{A_{\infty}} id )$, where $W(k_{\mathfrak{p}})$ is viewed as an $A_\infty$-algebra via $W(\phi_{\mathfrak{p}})\circ i(\Sigma )$. From now on $\Phi_{\mathfrak{p}} = W(\phi_{\mathfrak{p}})\circ i(\Sigma)$ where $\phi_{\mathfrak{p}} : A^{\text{perf}}_\circ \to k_{\mathfrak{p}}$ is a ring homomorphism such that $k_{\mathfrak{p}}$ is a perfect field and $\text{Ker} (\phi_{\mathfrak{p}}) = \mathfrak{p}$.       We have the following:

\begin{theorem}[Newton-Hodge Filtration, \cite{Katz}, Thm. 2.4.2]\label{NHFilt} Let $(M, \nabla , F_\Sigma)$ be an $F$-crystal over $A_\circ$ of rank $n$.  Suppose that $(A, B)$ is a break point of the Newton polygon of $(M, \nabla , F_\Sigma)^{(\Phi_{\mathfrak{p}})}$ at every point $\mathfrak{p}\in\Spec A_\circ$, and that $(A, B)$ lies on the Hodge polygon of $(M, \nabla , F_\Sigma)^{(\Phi_{\mathfrak{p}})}$ at every point $\mathfrak{p}$ of $\Spec A_\circ$.  Then there exists a unique $F_\Sigma$-stable horizontal $A_\infty$-submodule $M_1 \subset M$ with $M_1$ locally free of rank $A$ and $M/M_1$ locally free of rank $n- A$, such that: 

\begin{list}{(\arabic{marker})}{\usecounter{marker}}
        \item At every point $\mathfrak{p}\in\Spec A_\circ $ the Hodge and Newton slopes of \\ $(M_1 , \nabla |_{M_1}, F_\Sigma|_{M_1})^{(\Phi_{\mathfrak{p}})}$ are the first $A$ Hodge and Newton slopes of \\ $(M, \nabla , F_\Sigma)^{(\Phi_{\mathfrak{p}})}$.
        \item At every point $\mathfrak{p}\in\Spec A_\circ$ the Hodge and Newton slopes of \\ $(M/M_1 , \nabla |_{M/M_1}, F_\Sigma|_{M/M_1})^{(\Phi_{\mathfrak{p}})}$ are the last $n-A$ Hodge and Newton slopes of $(M, \nabla , F_\Sigma)^{(\Phi_{\mathfrak{p}})}$.
\end{list}

\end{theorem}

\section{Self-dual $F$-crystals over $\F_p$-algebras}

In this section $A_\circ$ is the same type of ring as at the end of the previous section.  We will define the notion of a self-dual $F$-crystal over $\F_p$.  Then we will prove a version of the Newton-Hodge Filtration Theorem (\ref{NHFilt}) for self-dual $F$-crystals over an $\F_p$-algebra $A_\circ$.  

Before we give the definition of a self-dual $F$-crystal we note that for a crystal $M$ over $A_\circ$ we can define its dual by requiring that \\ $M^*(B,I,\ga ; s) = Hom_B(M(B,I,\ga; s),B)$.  The base change condition is satisfied because all of the modules we are dealing with are locally free.  

\begin{definition} A self-dual $F$-crystal is a quintuple $(M , F_M, F_{M^*}, \Psi , c)$ where
\begin{list}{(\arabic{marker})}{\usecounter{marker}}
\item $c\in W(A_{\circ\circ})$ such that $c = p^m\cdot u$ where $u\in W(A_{\circ\circ})$ is a unit.
\item $m \geq 0$.
\item $(M,F_M)$ and $(M^* ,F_{M^*} )$ are $F$-crystals over $A_\circ$.
\item We have ${F_M}^*\circ F_{M^*} = c\cdot id_{M^*}$ and $F_{M^*}\circ {F_M}^* = c\cdot id_{(M^*)^{(\sigma )}}$ (note that here $F_M^* : M^* \to (M^{(\sigma)})^*$ is the transpose of $F_M$.  To make the required compositions we are using that $(M^{(\sigma)})^*$ and $(M^*)^{(\sigma)}$ are canonically isomorphic to each other.
\item $\Psi : (M, F) \to (M^* , F_{M^*})$ is an isomorphism of $F$-crystals (i.e. an isomorphism at every test object).
\end{list}
\end{definition}

Sometimes we will think of a self-dual $F$-crystal over $A_\circ$ under the equivalence of categories discussed in the previous section.  In this case we will write  $(M ,{F_M}_{\Sigma}, \nabla , {F_{M^*}}_{\Sigma}, \Psi , c)$ for  $(M , F_M, F_{M^*}, \Psi , c)$.  We see that under the equivalence, $\Psi: M\times M \to A_{\infty}$ is a non-degenerate bilinear form.
  
Given a self-dual crystal $(M , {F_M}_{\Sigma} , \nabla, {F_{M^*}}_{\Sigma}, \Psi , c)$ we can consider for each $\mathfrak{p}\in\Spec A_\circ$ the inverse image $\Phi_{\mathfrak{p}}: A_{\infty}\to W(k_{\mathfrak{p}})$ of $(M, {F_M}_\Sigma)$, $(M^* , {F_{M^*}}_\Sigma)$ and $\Psi$.  This results in a self-dual $F$-crystal over $k_{\mathfrak{p}}$ which we will denote \\ $(M , {F_M}_{\Sigma} , \nabla, {F_{M^*}}_{\Sigma}, \Psi , c)^{(\Phi_{\mathfrak{p}})}$.


Before we state the version of this theorem for self-dual crystals we will prove a lemma that is needed both for the proof of the theorem above and for the proof of the Newton-Hodge Filtration for self-dual $F$-crystals.

\begin{lemma}  If $a\in A_{\infty}$ is such that $\Phi_{\mathfrak{p}}(a) = 0$ for all $\mathfrak{p}$ then we have $a = 0$.
\end{lemma}
\begin{proof} Recall that $\Phi_{\mathfrak{p}} =W(\phi_{\mathfrak{p}})\circ  i(\Sigma)$.  Consider $i(\Sigma)(a) = (a_0 , a_1, \ldots )\in W(A_\circ^{\text{perf}})$.  We have $\Phi_{\mathfrak{p}} (a) = (\phi_{\mathfrak{p}}(a_0),\phi_{\mathfrak{p}}(a_1),\ldots) $.  Consider $a_i$.  If $\Phi_\mathfrak{p} (a) =0$ then $\phi_{\mathfrak{p}} (a_i)=0$ for each $\mathfrak{p}\in \Spec A_\circ^{\text{perf}}$.  Thus we have $a_i \in\mathfrak{p}$ for each $\mathfrak{p}\in A_\circ^{\text{perf}}$.  This means $a_i$ is nilpotent and hence $a_i =0$ for each $i$ since $A_\circ^{\text{perf}}$ is reduced.  This means $i(\Sigma)(a_i) = 0$, but $i(\Sigma)$ was injective so $a=0$ as required.
\end{proof}

\begin{theorem} [Newton Hodge Filtration for Self-Dual $F$-crystals]\label{NHFiltS} Let \\ $(M , {F_M}_{\Sigma} , \nabla, {F_{M^*}}_{\Sigma}, \Psi , c)$  be a self-dual $F$-crystal over $A_\circ$ of rank $n$.  Suppose that $(A, B)$ is a breakpoint of the Newton polygon of $(M , {F_M}_{\Sigma} , \nabla, {F_{M^*}}_{\Sigma}, \Psi , c)^{(\Phi_{\mathfrak{p}})}$ at every point $\mathfrak{p}\in\Spec A_\circ$, and that $(A, B)$ lies on the Hodge polygon of \\ $(M , {F_M}_{\Sigma} , \nabla, {F_{M^*}}_{\Sigma}, \Psi , c)^{(\Phi_{\mathfrak{p}})}$ at every point $\mathfrak{p}$ of $\Spec A_\circ$ and suppose furthermore that $A< \frac{n}{2}$.  Then there exist two $F_\Sigma$-stable horizontal $A_\infty$-submodules $M_1 \subset M_2 \subset M$ such that:

\begin{list}{(\arabic{marker})}{\usecounter{marker}}
   \item At every point $\mathfrak{p}\in\Spec A_\circ$ the Hodge and Newton slopes of \\ $(M_2/M_1 , \nabla |_{M_2/M_1}, F_\Sigma|_{M_2/M_1})^{(\Phi_{\mathfrak{p}})}$ are the middle $n-2A$ Hodge and Newton slopes of $(M, \nabla , F_\Sigma)^{(\Phi_{\mathfrak{p}})}$.
   \item  $(M_2/M_1 , \nabla |_{M_2/M_1}, F_\Sigma|_{M_2/M_1})$ can be given the structure of a self-dual $F$-crystal.
         \item  $M_1$ and $M/M_2$ are dual to each other, in other words $M_1 \cong (M/M_2)^*$ and $M/M_2 \cong M_1^*$.
\end{list}

\end{theorem}

\begin{proof} Note that since we know pointwise that $(A,B)$ is a breakpoint of the Newton polygon and lies on the Hodge polygon of $(M, \nabla, F_\Sigma)$, the same is true for $(n-A, (\frac{n}{2}-A)m -B)$ because of the symmetry of the Hodge and Newton slopes.  Thus we can apply Theorem~\ref{NHFilt} to the $F$-crystal  $(M, \nabla , F_\Sigma)$ at the breakpoint $(n-A, (\frac{n}{2}-A)m -B)$ to get a locally free $F_\Sigma$-stable $A_\infty$-submodule $0\subset M_2 \subset M$ such that $M_2$ satisfies the conclusion of Theorem 5.4.  In particular, at every point $\mathfrak{p}$ of $\Spec A_\circ$, $(A, B)$ is a breakpoint of the Newton polygon of $(M_2 , \nabla |_{M_2}, F_\Sigma|_{M_2})^{(\Phi_{\mathfrak{p}})}$ that lies on the Hodge polygon of $(M_2 , \nabla |_{M_2}, F_\Sigma|_{M_2})^{(\Phi_{\mathfrak{p}})}$.  Thus we can apply Theorem 5.4 again to $(M_2 , \nabla |_{M_2}, F_\Sigma|_{M_2})$ to obtain $0\subset M_1\subset M_2$, a locally free $F_\Sigma$-stable $A_\infty$-module of rank $A$ which satisfies the conclusions of Theorem 5.4.  We note that $M/M_1$ is locally free of rank $n-A$ since $M_1$, $M_2/M_1$, $M/M_2$ and $M$ are all locally free. To summarize we have: 

\begin{list}{(\arabic{marker})}{\usecounter{marker}}
   \item $M_1$ is locally free of rank $A$, $M_2$ is locally free of rank $n-A$, $M/M_1$ is locally free of rank $n-A$, $M/M_2$ is locally free of rank $A$, and $M_2/M_1$ is locally free of rank $n-2A$.
   \item At every point $\mathfrak{p}\in\Spec A_\circ $ the Hodge and Newton slopes of \\ $(M_1 , \nabla |_{M_1}, F_\Sigma|_{M_1})^{(\Phi_{\mathfrak{p}})}$ are the first $A$ Hodge and Newton slopes of \\ $(M, \nabla , F_\Sigma)^{(\Phi_{\mathfrak{p}})}$.
   \item At every point $\mathfrak{p}\in\Spec A_\circ $ the Hodge and Newton slopes of \\ $(M_2 , \nabla |_{M_2}, F_\Sigma|_{M_2})^{(\Phi_{\mathfrak{p}})}$ are the first $n-A$ Hodge and Newton slopes of \\ $(M, \nabla , F_\Sigma)^{(\Phi_{\mathfrak{p}})}$.
   \item At every point $\mathfrak{p}\in\Spec A_\circ$ the Hodge and Newton slopes of \\ $(M_2/M_1 , \nabla |_{M_2/M_1}, F_\Sigma|_{M_2/M_1})^{(\Phi_{\mathfrak{p}})}$ are the middle $n-2A$ Hodge and Newton slopes of $(M, \nabla , F_\Sigma)^{(\Phi_{\mathfrak{p}})}$.
\end{list}
We need to show that  $(M_2/M_1 , \nabla |_{M_2/M_1}, F_\Sigma|_{M_2/M_1})$ can be given a non-degenerate bilinear form.  We will show that $\Psi$ induces an isomorphism $\overline{\Psi}: M_2/M_1 \to (M_2/M_1)^*$. We first note that for every point $\mathfrak{p}\in\Spec A_\circ$ the self-dual $F$-crystal  $(M , {F_M}_{\Sigma} , \nabla, {F_{M^*}}_{\Sigma}, \Psi , c)^{(\Phi_{\mathfrak{p}})}$ over $k_{\mathfrak{p}}$ satisfies the hypotheses of Theorem 4.2.  This gives us a filtration $0\subset (M^{(\Phi_{\mathfrak{p}})})_1\subset (M^{(\Phi_{\mathfrak{p}})})_2\subset M^{(\Phi_{\mathfrak{p}})}$ where $\Psi^{(\Phi_{\mathfrak{p}})}(m_1)(m_2) = 0$ for all $m_1\in (M^{(\Phi_{\mathfrak{p}})})_1$, $m_2\in (M^{(\Phi_{\mathfrak{p}})})_2$.   We claim that $(M^{(\Phi_{\mathfrak{p}})})_1 = M_1^{(\Phi_{\mathfrak{p}})}$.  Note that $M_1^{(\Phi_{\mathfrak{p}})}$ is ${F_M}^{(\Phi_{\mathfrak{p}})}_\Sigma$-stable and that $M^{(\Phi_{\mathfrak{p}})}/M_1^{(\Phi_{\mathfrak{p}})}$ is free.  Also the Newton slopes of $( M_1^{(\Phi_{\mathfrak{p}})}, {F_M}^{(\Phi_{\mathfrak{p}})}_\Sigma|_{ M_1^{(\Phi_{\mathfrak{p}})}})$ are  the first $A$ Newton slopes of $(M^{(\Phi_{\mathfrak{p}})}, {F_M}^{(\Phi_{\mathfrak{p}})}_\Sigma )$. But $(M^{(\Phi_{\mathfrak{p}})})_1$ is the unique submodule of $M^{(\Phi_{\mathfrak{p}})}$ that satisfies these properties by Theorem 4.1, so we must have $(M^{(\Phi_{\mathfrak{p}})})_1 = M_1^{(\Phi_{\mathfrak{p}})}$. Similarly,   $(M^{(\Phi_{\mathfrak{p}})})_2 = M_2^{(\Phi_{\mathfrak{p}})}$. Now given $m_1\in M_1$ and $m_2 \in M_2$ we have that $\Phi_{\mathfrak{p}}(\Psi(m_1)(m_2)) = \Psi^{(\Phi_{\mathfrak{p}})}(m_1\otimes 1)(m_2 \otimes 1)$.  We also know that $m_1 \otimes 1 \in  (M^{(\Phi_{\mathfrak{p}})})_1$ and $m_2 \otimes 1 \in  (M^{(\Phi_{\mathfrak{p}})})_2$.  Thus  $\Phi_{\mathfrak{p}}(\Psi(m_1)(m_2)) = 0$ for each $\mathfrak{p}\in \Spec A_\circ$ so we have $\Psi(m_1)(m_2)=0$ for all $m_1\in M_1$ and $m_2 \in M_2$.  Similarly, $\Psi(m_2)(m_1) = 0$ for all $m_1\in M_1$ and $m_2 \in M_2$; in other words, $M_1$ and $M_2$ are perpendicular. 

To define $\overline{\Psi}$ first we will show that $\Psi(M_2) = (M/M_1)^*$.  Let $\text{res}_{M_1}: M^* \to M_1^*$ be the restriction map, i.e. $\text{res}_{M_1}(\phi) = \phi|_{M_1}$.  We claim that $\text{res}_{M_1}$ is surjective.  All we need to show is that this is true locally, but  $M_1$ and $M/M_1$ are both locally free so we have that $M = M_1 \oplus M/M_1$ locally.  Thus, locally, $M^* \twoheadrightarrow M_1^*$ and $\text{res}_{M_1}$ is surjective.  Now consider $\text{res}_{M_1}\circ \Psi$.  We have $\Psi(m_2)(m_1) = 0$ for all $m_1 \in M_1$, $m_2\in M_2$.  Thus $\text{res}_{M_1}\circ \Psi$ factors through the quotient $M/M_2$; we will denote the induced map $\Psi ': M/M_2 \to M_1^*$.  We have \\
\centerline{\xymatrix { 
M \ar@{->>}[r]^{\Psi} \ar@{->>}[rd] & M^* \ar@{->>}[r]^{\text{res}_{M_1}}& M_1^*\\
 & M/M_2 \ar[ru]_{\Psi '} & },} 
\\
so $\Psi '$ is surjective.  Since $M/M_2$ and $M_1^*$ are both locally free and of rank $A$ we have that $\Psi '$ is in fact an isomorphism.  Recall that $(M/M_1)^* = \text{Ker}( \text{res}_{M_1})$.  We have $\Psi(M_2)\subset (M/M_1)^*$ since $M_1$ and $M_2$ are perpendicular to each other.  To show the other containment suppose $\phi \in (M/M_1)^*$.  Then there exists $m\in M$ such that $\Psi(m) = \phi$.  We have $\Psi ' (m + M_2) = \text{res}_{M_1} (\Psi(m)) = 0$.  But $\Psi '$ is injective, so $m\in M_2$ as required.  We also see that $\Psi'$ provides an isomorphism between $M/M_2$ and $M_1^*$ which shows that $M_1$ and $M/M_2$ are dual to each other.

Now we show that $\Psi$ induces an isomorphism $M_2/M_1 \to (M_2/M_1)^*$.  Note that $M^*\twoheadrightarrow M_2^*$ since this is the case locally.  This gives us $\text{res}_{M_2}: (M/M_1)^* \twoheadrightarrow (M_2/M_1)^*$.  Let $\Psi|_{M_2}$ denote $\Psi$ restricted to $M_2$. We see that $(\text{res}_{M_2} \circ\Psi |_{M_2})(m_1) = 0$ since $M_1$ and $M_2$ are perpendicular.  Thus $\text{res}_{M_2} \circ\Psi |_{M_2}$ factors through the quotient $M_2/M_1$, so we get an induced map $\overline{\Psi}: M_2/M_1 \to (M_2/M_1)^*$.  We have \\
\centerline{\xymatrix { 
M_2 \ar@{->>}[r]^{\Psi|_{M_2}} \ar@{->>}[rd] & (M/M_1)^* \ar@{->>}[r]^{\text{res}_{M_2}}& (M_2/M_1)^*\\
 & M_2/M_1 \ar[ru]_{\overline{\Psi}} & .}}  
Since $\overline{\Psi}$ is surjective, and since $M_2/M_1$  and $(M_2/M_1)^*$ are both locally free of rank $n-2A$,  $\overline{\Psi}$ is an isomorphism as required.



\end{proof}

\end{document}